\documentclass[a4paper,12pt]{article}

\usepackage{amssymb}
\usepackage{epsfig}
\usepackage{amsfonts}
\usepackage{amsmath}
\usepackage{euscript}
\usepackage{amscd}
\usepackage{amsthm}
\DeclareMathAlphabet{\mathpzc}{OT1}{pzc}{m}{it}

\newtheorem{thm}{Theorem}[section]
\newtheorem{lem}[thm]{Lemma}
\newtheorem{prop}[thm]{Proposition} 
\newtheorem{cor}[thm]{Corollary}
\newtheorem{rem}[thm]{Remark}

\newcommand{\m}{\mathpzc{m}}
\newcommand{\bZ}{\mathbb Z}

\newcommand{\A}{\mathbb A}
\newcommand{\V}{\mathbb V}
\newcommand{\G}{\mathbb G}

\newcommand{\td}{\operatorname{tr.deg}}
\newcommand{\w}{\operatorname{wt}}
\newcommand{\gr}{\operatorname{gr}}

\title{On the Cancellation Problem for the Affine Space $\A^3$ \\
in characteristic $p$}
\author{{Neena Gupta\footnote{\it
{Current address: 
Stat-Math Unit, Indian Statistical Institute, 203 B.T. Road, Kolkata 700108, India.
e-mail addresses: neenag@isical.ac.in, rnanina@gmail.com}}}\\
{\small{\it School of Mathematics, Tata Institute of Fundamental Research}} \\
{\small{\it Dr. Homi Bhabha Road, Colaba, Mumbai 400005, India}}\\
{\small{\it e-mail : neena@math.tifr.res.in}}}

\begin{document}

\date{}
\maketitle

\begin{abstract}
We show that the Cancellation Conjecture 
does not hold
for the affine space $\A^3_k$ over any field $k$ of positive characteristic.
We prove that an example of T. Asanuma provides a three-dimensional $k$-algebra $A$
for which $A$ is not isomorphic to $k[X_1,X_2,X_3]$ although $A[T]$ is isomorphic to $k[X_1, X_2, X_3, X_4]$.

\medskip

\end{abstract}

\section{Introduction}
Let $k$ be an algebraically closed field.
The long-standing Cancellation Problem for Affine Spaces 
(also known as Zariski Problem) asks: 
if $\V$ is an affine $k$-variety such that $\V \times \A^1_k \cong \A^{n+1}_k$,
does it follow that $\V \cong \A_k^n$? Equivalently, 
if $A$ is an affine $k$-algebra
such that $A[X]$ is isomorphic to the polynomial ring $k[X_1, \dots, X_{n+1}]$,
does it follow that $A$ is isomorphic to $k[X_1, \dots, X_n]$?

For $n=1$, a positive solution to the problem was given by 
Shreeram S. Abhyankar, Paul Eakin and William J. Heinzer
(\cite{AEH}). For $n=2$, a positive solution to the problem was given by 
Takao Fujita, Masayoshi Miyanishi and Tohru Sugie
(\cite{F}, \cite{MS}) in characteristic zero and 
Peter Russell (\cite{R}) in positive characteristic;  a simplified algebraic proof was given by 
Anthony J. Crachiola and Leonid G. Makar-Limanov
in \cite{CM}. The problem remained open for $n>2$.

In this paper, we shall show that for $n=3$, a threefold investigated by Teruo Asanuma in
\cite{A} and \cite{A2} gives a negative solution to the Cancellation Problem in positive
characteristic (Corollary \ref{cor}). For convenience, we shall use the notation $R^{[n]}$ for a polynomial
ring in $n$ variables over a ring $R$.

In \cite{A}, Asanuma has proved the following theorem (cf. \cite[Theorem 5.1, Corollary 5.3]{A}, \cite[Theorem 1.1]{A2}).

\begin{thm}\label{Ae}
 Let $k$ be a field of characteristic $p(>0)$ and
 $$ A= k[X,Y,Z,T]/(X^m Y + Z^{p^e} + T+ T^{sp})$$
 where $m, e, s$ are positive integers such that $p^e \nmid sp$
 and $sp \nmid p^e$. Let $x$ denote the image of $X$ in $A$. 
 Then $A$ satisfies the following two properties:
 \begin{enumerate}
  \item $A^{[1]} \cong_{k[x]} k[x]^{[3]} \cong_k k^{[4]}$.
  \item $A \ncong_{k[x]} k[x]^{[2]}$.
 \end{enumerate}
\end{thm}

In \cite[Theorem 2.2]{A2}, Asanuma used the above example to construct non-linearizable 
algebraic torus actions on $\A^n_k$ over any infinite field $k$ of positive characteristic
when $n \ge 4$. The following problem occurs in the same paper (see \cite[Remark 2.3]{A2}).

\medskip

\noindent
{\bf Question.} Let $A$ be as in Theorem \ref{Ae}. Is $A \cong_k k^{[3]}$?

\medskip

In section 3 of this paper we shall use techniques developed by L. Makar-Limanov
and A. Crachiola to show (Theorem \ref{ce}) 
that the answer to the above problem is negative
when $m>1$. Thus, in view of Theorem \ref{Ae}, 
we get counter-examples to the Cancellation Problem in positive 
characteristic for $n=3$ (Corollary \ref{cor}).

\section{Preliminaries}

{\bf Definition.}
Let $A$ be a $k$-algebra and let $\phi: A \to A^{[1]}$ be a 
$k$-algebra homomorphism. For an indeterminate $U$ over $A$, let the notation 
$\phi_U$ denote the map $\phi: A \to A[U]$.  
$\phi$ is said to be an {\it exponential map on $A$} if $\phi$ satisfies the following two properties:
\begin{enumerate}
 \item [\rm (i)] $\varepsilon_0 \phi_U$ is identity on $A$, where 
 $\varepsilon_0: A[U] \to A$ is the evaluation at $U = 0$.
 \item[\rm (ii)] $\phi_V \phi_U = \phi_{V+U}$, where 
 $\phi_V: A \to A[V]$ is extended to a homomorphism 
 $\phi_V: A[U] \to A[V,U]$ by  setting $\phi_V(U)= U$.
 \end{enumerate}
When $A$ is the coordinate ring of an affine $k$-variety, any 
action of the additive group $\G_a:= (k, +)$ on the variety 
corresponds to an exponential map $\phi$, the two axioms of a group action on a set translate
into the conditions (i) and (ii) above. 

The ring of $\phi$-invariants of an exponential map $\phi$ on $A$
is a subring of $A$ given by 
$$
A^{\phi} = \{a \in A\,| \,\phi (a) = a\}.
$$
An exponential map $\phi$ is said to be non-trivial if $A^{\phi} \neq A$.

The {\it Derksen invariant} of a $k$-algebra $A$, 
denoted by ${\rm DK} (A)$,
is defined to be the subring of $A$ generated by the $A^{\phi}$s, where 
$\phi$ varies over the set of non-trivial exponential maps on $A$.

We summarise below some useful properties of an exponential map $\phi$.

\begin{lem}\label{exp3}
Let $A$ be an affine domain over a field $k$. Suppose that there exists a non-trivial
exponential map $\phi$ on $A$. Then the following statements hold:
 \begin{enumerate}
\item [\rm (i)] $A^{\phi}$ is factorially closed in $A$, i.e., 
if $a, b \in A$ such that $0\neq ab \in A^{\phi}$, 
then $a, b \in A^{\phi}$.
\item [\rm (ii)] $A^{\phi}$ is algebraically closed in $A$.
\item [\rm (iii)] $\td_k (A^{\phi}) = \td_k (A) -1$. 
\item[\rm (iv)] There exists $c \in A^{\phi}$ such that $A[c^{-1}]= A^{\phi}[c^{-1}]^{[1]}$. 
\item [\rm (v)] If $\td_k (A)=1$ then  $A= \tilde{k}^{[1]}$, where $\tilde{k}$ is the algebraic closure 
of $k$ in $A$ and $A^{\phi} = \tilde{k}$.
\item[\rm (vi)] Let $S$ be a multiplicative subset of $A^{\phi}\setminus \{0\}$. 
Then $\phi$ extends 
to a non-trivial exponential map $S^{-1}\phi$ on $S^{-1}A$ by setting 
$(S^{-1}\phi) (a/s) = \phi(a)/s$ for $a \in A$ and $s \in S$. 
Moreover, the ring of invariants of $S^{-1}\phi$ is $S^{-1}(A^{\phi})$.
\end{enumerate}
\end{lem}

\begin{proof}
Statements (i)--(iv) occur in \cite[p. 1291--1292]{C};
(v) follows from (ii), (iii) and (iv); 
(vi) follows from the definition.
\end{proof}

Let $A$ be an affine domain over a field $k$.
$A$ is said to have a {\it proper $\bZ$-filtration} if
there exists a collection of $k$-linear subspaces $\{A_n\}_{n \in \bZ}$ of $A$ satisfying 
\begin{enumerate}
\item [\rm (i)] $A_n \subseteq A_{n+1}$ for all $n \in \bZ$,
\item [\rm (ii)]  $A= \bigcup_{n \in \bZ} A_n$,
\item [\rm (iii)] $\bigcap_{n\in \bZ} A_n = (0)$ and 
\item [\rm (iv)]  $(A_n\setminus A_{n-1}). (A_m \setminus A_{m-1}) 
\subseteq A_{n+m} \setminus A_{n+m-1}$ for all $n, m \in \bZ$.
\end{enumerate}
Any proper $\bZ$-filtration on $A$ determines the following $\bZ$-graded integral domain
$$
\gr (A) : = \bigoplus_i A_i/A_{i-1}.
$$
There exists a map
$$
\rho: A \to \gr (A)~~ {\text{defined by}}~~ \rho(a) = a + A_{n-1}, ~~{\rm if}~~ a \in A_n \setminus A_{n-1}.
$$

We shall call a proper $\bZ$-filtration $\{A_n\}_{n \in \bZ}$ of $A$ 
to be {\it admissible} if there exists a finite generating set $\Gamma$
of $A$ such that, for any $n \in \bZ$ and $a \in A_n$, $a$ can be written
as a finite sum of monomials in elements of $\Gamma$ and each of these monomials 
are elements of $A_n$. 

\begin{rem}\label{grmap}
 {\em 
(1) Note that $\rho$ is not a ring homomorphism.  For instance,
if $i <n $ and $ a_1, \cdots, a_{\ell} \in A_{n}\setminus A_{n-1}$
such that $a_1+ \cdots+ a_{\ell} \in A_i\setminus A_{i-1} (\subseteq A_{n-1})$,
then $\rho (a_1+ \cdots+ a_{\ell}) \neq 0$ but 
$\rho (a_1) +\cdots + \rho (a_{\ell}) = a_1+ \cdots+ a_{\ell} + A_{n-1}= 0$ in $\gr (A)$.

(2) Suppose that $A$ has a proper $\bZ$-filtration and 
a finite generating set $\Gamma$ which makes the filtration admissible.
Then $\gr(A)$ is generated by $\rho(\Gamma)$, since if
$ a_1, \cdots, a_{\ell}$ and $a_1+ \cdots+ a_{\ell} \in A_n\setminus A_{n-1}$,
then $\rho(a_1+ \cdots+ a_{\ell}) = \rho (a_1) +\cdots + \rho (a_{\ell})$
and $\rho(ab)= \rho(a)\rho(b)$ for any $a, b \in A$.

(3) Suppose that $A$ has a $\bZ$-graded algebra structure, 
say, $A = \bigoplus_{i \in \bZ} C_i$. Then 
there exists a proper $\bZ$-filtration $\{A_n\}_{n \in \bZ}$ on $A$ 
defined by $A_n:= \bigoplus_{i \le n} C_i$.
Moreover, $\gr (A)  = \bigoplus_{n \in \bZ} A_n/A_{n-1} \cong \bigoplus_{n \in \bZ} C_n= A$
and, for any element $a \in A$,  the image of $\rho (a)$ under the isomorphism 
$\gr (A)  \to A$ is the homogeneous component of $a$ in $A$ of maximum degree.
If $A$ is a finitely generated $k$-algebra,
then the above filtration on $A$ is admissible.    
}
\end{rem}

The following version of a result of H. Derksen, O. Hadas and L. Makar-Limanov
is presented in \cite[Theorem 2.6]{C}.

\begin{thm}\label{MDH}
Let $A$ be an affine domain over a field $k$ with a proper 
$\bZ$-filtration which is admissible.
Let $\phi$ be a non-trivial exponential map on $A$.
Then $\phi$ induces a non-trivial exponential map $\bar{\phi}$ on 
$\gr (A)$ such that $\rho( {A^{\phi}}) \subseteq {\gr (A)}^{\bar{\phi}}$.
\end{thm}

The following observation is crucial to our main theorem.

\begin{lem}\label{r1}
Let $A$ be an affine $k$-algebra such that $\td_k A >1$. If ${\rm DK} (A) \subsetneqq A$, then $A$ 
is not a polynomial ring over $k$.
\end{lem}

\begin{proof}
Suppose that $A = k[X_1, \dots, X_n]$, a polynomial ring in $n( >1)$ variables over $k$.
For each $i$, $1\le i\le n$, consider the exponential map 
$\phi_i: A \to A[U]$ defined by $\phi_i(X_j) = X_j + \delta_{ij} U$, where
$\delta_{ij} = 1$ if $j=i$ and zero otherwise.
Then $A^{\phi_i} = k[X_1, \dots, {X_{i-1}}, X_{i+1},  \dots, X_n]$. 
It follows that ${\rm DK} (A) = A$. Hence the result.  
\end{proof}

\section{Main Theorem}

In this section we will prove our main result (Theorem \ref{ce}, Corollary \ref{cor}).
We first prove a few lemmas on the existence of exponential maps on certain affine domains.

Let $L$ be a field with algebraic closure $\bar{L}$.
An $L$-algebra $D$ is said to be {\it geometrically integral}
if $D \otimes_L \bar{L}$ is an integral domain. 
It is easy to see that if an $L$-algebra $D$ is geometrically integral 
then $L$ is algebraically closed in $D$. We therefore have the following lemma.

\begin{lem}\label{geo}
Let $L$ be a field and $D$ a geometrically integral $L$-algebra
such that $\td_L D =1$. If $D$ admits a non-trivial exponential map
then $D= L^{[1]}$. 
\end{lem}

\begin{proof}
 Follows from Lemma \ref{exp3} (v) and the fact that $L$ is algebraically closed in $D$.
\end{proof}

\begin{lem}\label{prime1}
Let  $L$ be a field of characteristic $p >0$ and $Z, T$ be two indeterminates over $L$. 
Let $g = Z^{p^e} -\alpha T^m +\beta$, where $0 \neq \alpha, 0\neq \beta \in L$, 
$e\ge 1$ and $m>1$ are integers such that $m$ is coprime to $p$. Then
$D = L[Z, T]/(g)$ does not admit any non-trivial exponential map.
\end{lem}

\begin{proof}
We first show that $D$ is a geometrically integral $L$-algebra. 
Let $\bar{L}$ denote the algebraic closure of $L$.
Since $m$ is coprime to $p$, the polynomial $\alpha T^m -\beta$ 
is square-free in $\bar{L}[T]$ and hence, by Eisenstein's criterion, $g$ is an
irreducible polynomial in $\bar{L}[T][Z]$. Thus $D\otimes_L \bar{L}$ is an integral domain.
Let $\lambda$ be a root of the polynomial $Z^{p^e}+\beta \in \bar{L}[Z]$
and let $M = (T, Z-\lambda)\bar{L}[Z, T]$.
Then $M$ is a maximal ideal of $\bar{L}[Z, T]$ such that $g \in M^2$.
It follows that  $D\otimes_L \bar{L} (= \bar{L}[Z, T]/(g))$ is not a normal domain;
in particular, $D\neq L^{[1]}$. Hence the result, by Lemma \ref{geo}. 
\end{proof}

\begin{lem}\label{lem1}
Let $B$ be an affine domain over an infinite field $k$. 
Let $f \in B$ be such that  $f - \lambda $ is a prime element of $B$ 
for infinitely many $\lambda \in k$.  Let $\phi: B \to B[U]$ be 
a non-trivial exponential map on $B$ such that
$f \in B^{\phi}$. Then there exists $\beta \in k$ such that  $f  - \beta$ 
is a prime element of $ B$ and $\phi$
induces a non-trivial exponential map on $B/(f - \beta)$.   
\end{lem}

\begin{proof}
Let $x \in B \setminus B^{\phi}$. Then
$\phi (x) =  x + a_1 U + \cdots + a_{n}U^{n}$ for some $n \ge 1$, 
$a_1, \cdots, a_n \in B$ with $a_n \neq 0$. Since $B$ is affine and $k$
is infinite, there exists $\beta \in k$ such that  $f - \beta$ 
is a prime element of $B$ and $a_n \notin (f -\beta)B$. 
Let $\bar{x}$ denote the image of $x$ in $B/(f-\beta)$. 
Since $\phi(f-\beta) =f-\beta$, $\phi$ induces an exponential map ${\phi_1}$ on $B/(f-\beta)$ 
which is non-trivial since ${\phi_1}(\bar{x}) \neq \bar{x}$ by the choice of $\beta$. 
\end{proof}

\begin{lem}\label{lem3}
(i) Let $B = k [X,Y, Z,T] /(X^m Y - g)$ where $m >1$ and $g \in k[Z, T]$.
Let $f$ be a prime element of $k[Z,T]$ such that $g \notin fk[Z, T]$ and let $D= B/fB$. 
Then $D$ is an integral domain. 

(ii) Supppose that there exists a maximal ideal $M$ of $k[Z,T]$ such that $ f \in M$ and
$g \in fk[Z, T] + M^2$. Then there does not exist any non-trivial exponential map $\phi$ on 
$D$ such that the image of $Y$ in $D$ belongs to $D^{\phi}$. 
\end{lem}

\begin{proof}
(i) Let $\bar{g}$ denote the image of $g$ in the integral domain $C:= k[Z,T]/(f)$.             
Since $g \notin fk[Z, T]$, it is easy to see that 
$D = B/fB = C [X, Y]/(X^m Y - \bar{g})$ is an integral domain.

(ii) Let $\alpha \in k[Z, T]$ be such that $g-\alpha f \in M^2$ and let $\m = (M, X) k(Y)[X, Z, T]$.
Since $X^mY-g+\alpha f \in \m^2$, we see that 
$$
D \otimes_{k[Y]} k(Y) = k(Y)[X, Z, T]/ (X^mY-g, f) = k(Y)[X, Z, T]/ (X^mY-g+\alpha f, f)
$$
is not a normal domain. In particular, $D \otimes_{k[Y]} k(Y)$
is not a polynomial ring over a field. The result now follows from 
Lemma \ref{exp3} (vi) and (v). 
\end{proof}

We now come to the main technical result of this paper.

\begin{lem}\label{l1}
Let $k$ be any field of positive characteristic $p$.
Let $B= k[X,Y,Z,T]/(X^m Y + T^{s} + Z^{p^e})$, 
where $m, e, s$ are positive integers such that 
$s = p^r q$, $p \nmid q$, $q >1$, $e> r \ge 1$ and $m>1$.
Let $x,y, z, t$ denote the images of $X, Y, Z, T$ in $B$.
Then there does not exist any non-trivial exponential map
$\phi$ on $B$ such that $y \in B^{\phi}$.
\end{lem}

\begin{proof}
Without loss of generality we may assume that $k$ is algebraically closed.
We first note that $B$ has the structure of a $\bZ$-graded algebra over $k$
with the following weights on the generators:
$$
\w (x) = -1, ~~ \w (y) = m, ~~ \w (z)=0, ~~ \w (t)=0.
$$
This grading defines a proper $\bZ$-filtration on $B$ such that 
$\gr (B)$ is isomorphic to $B$ (cf. Remark \ref{grmap} (3)). 
We identify $B$ with $\gr(B)$.
For each $g \in B$, let $\hat{g}$ denote 
the image of $g$ under the composite map $B \to \gr(B) \cong B$.
By Theorem \ref{MDH}, $\phi$ induces a non-trivial
exponential map $\hat{\phi}$ on $B (\cong \gr (B))$ 
such that $\hat{g} \in B^{\hat{\phi}}$ whenever $g \in B^{\phi}$.
%Note that each homogeneous element of $B$ is of the form $x^iy^jg_1(z, t)$ where $g_1(z, t) \in k[z, t]$.  
Using the relation $x^my= -(z^{p^e} + t^{s})$ if necessary, 
we observe that each element $g \in B$
can be uniquely written as 
\begin{equation}\label{eq1}
 g = \sum_{n \ge 0} g_n( z, t ) x^n + \sum_{j >0,~0 \le i < m} g_{ij}(z, t) x^i y^j,
\end{equation}
where $g_n(z,t)$, $g_{ij}(z,t) \in k[z,t]$.
%, $ 0 \le i < m$. 
From the above expression and the weights defined on the generators, it follows that
each homogeneous element of $B$ is of the form $x^iy^jg_1(z, t)$, where $g_1(z, t) \in k[z, t]$.
%Thus $\hat{g}$ is either of the form $ x^ig_1(z,t) $ or $ x^ig_1(z,t) 
%Moreover, if $g \in k[x, z, t]$ then 
%\begin{equation}
%\hat{g}= x^ig_1(z,t) {\text{~~for some~~}} i \ge 0 {~~\text{and}~~} g_1(z,t) \in k[z,t] 
%\end{equation}
%and if $g \notin k[x, z, t]$ then 
%\begin{equation}
%\hat{g}= x^iy^jg_1(z,t) {\text{~~for some~~}} 0 \le i < m, j>0 {~~\text{and}~~} g_1(z,t) \in k[z,t]. 
%\end{equation}

%we observe that for $g \in B$, either
%\begin{equation}
%\hat{g}= x^ig_1(z,t) {\text{~~for some~~}} i \ge 0 {~~\text{and}~~} g_1(z,t) \in k[z,t] 
%\end{equation}
%or
%\begin{equation}
%\hat{g}= x^iy^jg_1(z,t) {\text{~~for some~~}} 0 \le i < m, j>0 {~~\text{and}~~} g_1(z,t) \in k[z,t]. 
%\end{equation}
We now prove the result by contradiction. 
Suppose, if possible, that $y \in B^{\phi}$.
Then $y (=\hat{y}) \in B^{\hat{\phi}}$.  

%We first see that $h(z,t) \in B^{\hat{\phi}}$ for some polynomial $h(z,t) \in k[z,t] \setminus k$. Since $\td_{k} (B^{{\hat{\phi}}}) =2$, there exists 
%a homogeneous element $g \in B^{\hat{\phi}}$ such that ${g} \notin k[y]$. Since ${g}$ is homogeneous, we have ${g} = x^iy^jg_1(z, t)$ for some $g_1(z, t) \in k[z, t]$.
%Then either $g_1(z,t) \notin k^*$ or $x \mid {g}$. If $g_1(z,t) \notin k^*$, then we set $h(z,t):= g_1(z,t)$ and if $g_1(z,t) \in k^*$ then we set $h(z,t):= z^{p^e} + t^{s}=-x^my$. 
%In either case, since ${g} \in B^{\hat{\phi}}$, we have $h(z,t) \in  B^{\hat{\phi}}$ by Lemma \ref{exp3} (i). Thus $k[y, h(z,t)] \subseteq B^{\hat{\phi}}$ for some $h(z,t) \in k[z,t] \setminus k$.

We first see that $h(z,t) \in B^{\hat{\phi}}$ for some 
polynomial $h(z,t) \in k[z,t] \setminus k$. 
Since $\td_{k} (B^{{\phi}}) =2$, one can see that there exists 
$g \in B^{{\phi}}$ such that $\hat{g} \notin k[y]$. 
Since $\hat{g}$ is homogeneous, 
we have $\hat{g} = x^iy^jg_1(z, t)$ for some $g_1(z, t) \in k[z, t]$.
Then either $g_1(z,t) \notin k^*$ or $x \mid \hat{g}$. 
If $g_1(z,t) \notin k^*$, then we set $h(z,t):= g_1(z,t)$ 
and if $g_1(z,t) \in k^*$ then we set $h(z,t):= z^{p^e} + t^{s}=-x^my$. 
In either case, since $\hat{g} \in B^{\hat{\phi}}$ by Theorem \ref{MDH},
we have $h(z,t) \in  B^{\hat{\phi}}$ by Lemma \ref{exp3} (i).
Thus $k[y, h(z,t)] \subseteq B^{\hat{\phi}}$ for some $h(z,t) \in k[z,t] \setminus k$.

Now consider another $\bZ$-graded structure on the $k$-algebra $B$
with the following weights assigned to the generators of $B$:
$$
\w (x) = 0, ~~ \w (y) = qp^{e}, ~~ \w (z)=q, ~~ \w (t)=p^{e-r}.
$$
This grading defines another proper $\bZ$-filtration on $B$ such that 
$\gr (B)$ is isomorphic to $B$. We identify $B$ with $\gr(B)$ as before.
For each $g \in B$, let $\overline{g}$ denote 
the image of $g$ under the composite map $B \to \gr(B) \cong B$.
Again by Theorem \ref{MDH}, $\hat{\phi}$ induces a non-trivial 
exponential map $\bar{\phi}$ on $B (\cong{\rm gr}(B))$ 
such that $k[{y}, \overline{h(z,t)}] \subseteq B^{\bar{\phi}}$.
Since $k$ is an algebraically closed field, we have 
$$
\overline{h(z,t)} = \theta{z}^i {t}^j \prod_{\ell \in \Lambda} ( {z}^{p^{e-r}} +\mu_{\ell} {t}^{q}),
$$ 
for some $\theta, \mu_{\ell} \in k^*$. 
Since $\overline{h(z,t)} \notin k$, 
it follows that there exists a prime factor $w$ of $\overline{h(z, t)}$ in $k[z,t]$. 
By above, we may assume that either $w= z$ or $w= t$ or $w = z^{p^{e-r}} + \mu t^q$
for some $\mu \in k^*$. By Lemma \ref{exp3} (i), $k[y, w] \subseteq B^{\bar{\phi}}$. 

We first show that $w \neq {z}^{p^{e-r}} + {t}^{q}$.
If $w= {z}^{p^{e-r}} + {t}^{q}$, then ${x}^m{y}=-({z}^{p^{e-r}} + {t}^{q})^{p^{r}}
= - w^{p^r} \in B^{\bar{\phi}}$
which implies that $x \in B^{\bar{\phi}}$ (cf. Lemma \ref{exp3} (i)). 
Let $L$ be the field of fractions of 
$C:= k[{x}, {y}, w](\hookrightarrow B^{\bar{\phi}})$.
By Lemma \ref{exp3} (vi), 
${\bar{\phi}}$ induces a non-trivial exponential map on 
$$
B\otimes_C L = k({x}, {y}, w)[{z}, {t}]
\cong \dfrac{L[Z,T]}{(Z^{p^{e-r}} + T^{q}-w)},
$$
which contradicts Lemma \ref{prime1}. 

So, we assume that either $w= z$ or $w= t$ or $w = z^{p^{e-r}} + \mu t^q$
for some $\mu \in k^*$ with $\mu \neq 1$.
Note that, for every $\lambda \in k^*$,  $w-\lambda$ is a prime element of $k[z, t]$ 
and $({z}^{p^e} + {t}^{qp^{r}}) \notin (w-\lambda)k[z, t]$ and 
hence, $(w-\lambda)$ is a prime element of $B$ (cf. Lemma \ref{lem3} (i)). 
Thus,  by Lemma \ref{lem1}, there exists a $\beta \in k^*$ such 
that $\bar{\phi}$ induces a non-trivial exponential map $\phi_1$ on $B/(w-\beta)$ with
the image of $y$ in $B/(w-\beta)$ lying in the ring of invariants of $\phi_1$.  
Set $f:= w-\beta$, $g_0:= z^{p^{e-r}} + t^q$ and
$g := {g_0}^{p^r} = {z}^{p^e} + {t}^{qp^{r}}$. 
Since $\mu \neq 1$, $f$ and $g_0$ are not comaximal in $k[z, t]$.
%We now define a maximal ideal $M$ of $k[z, t]$ as follows:  if $w= z$, then set $M: = (z-\beta, t-\nu)$, where $\nu \in k$ is such 
%that $\nu^{q}+\beta^{p^{e-r}} =0$; if $w=t$, then set $M: = (z-\epsilon, t-\beta)$, where $\epsilon \in k$ is such that $\epsilon^{p^{e-r}} + \beta^{q} =0$; and 
%if $w= z^{p^{e-r}} + \mu t^q$, then set $M: = (z-\gamma, t-\delta)$, where $\gamma, \delta \in k$ are such that $(1-\mu)\delta^q+\beta=0$ and $\gamma^{p^{e-r}} + \mu\delta^{q}-\beta =0$.
Let $M$ be a maximal ideal of $k[z, t]$ containing $f$ and $g_0$.   
Then $g = {g_0}^{p^r} \in M^2$ since $r\ge 1$.
This contradicts Lemma \ref{lem3} (ii).

Hence the result.
\end{proof}

\begin{prop}\label{p1}
Let $k$ be any field of characteristic $p$ ($> 0$) and
$$
A= k[X,Y,Z,T]/(X^m Y + Z^{p^e} + T+ T^{sp}),
$$
where $m, e, s$ are positive integers such that $p^e \nmid sp$, $sp \nmid p^e$ and $m >1$.
Let $\phi$ be a non-trivial exponential map on $A$.
Then $A^{\phi} \subseteq k[x, z, t]$, where $x, z, t, y$  denote the images of $X, Z, T, Y$ in $A$.
In particular, ${\rm DK} (A) \subsetneqq A$. 
\end{prop}

\begin{proof}
We first consider the $\bZ$-graded structure of $A$ with weights:
$$
\w (x) = -1, ~~ \w (y) = m, ~~ \w (z)=0, ~~ \w (t)=0.
$$
For each $g \in A$, let $\hat{g}$ denote 
the homogeneous component of $g$ in $A$ of maximum degree.
As in the proof of Lemma \ref{l1}, we see that
$\phi$ induces a non-trivial exponential map $\hat{\phi}$ on 
$A (\cong \gr (A))$ such that $\hat{g} \in A^{\hat{\phi}}$ whenever $g \in A^{\phi}$.
Using the relation $x^my= -(z^{p^e} +t + t^{sp})$ if necessary, 
we observe that each element $g \in A$
can be uniquely written as 
\begin{equation}\label{eq2}
 g = \sum_{n \ge 0} g_n( z, t ) x^n + \sum_{j >0,~0 \le i < m} g_{ij}(z, t) x^i y^j,
\end{equation}
where $g_n(z,t)$, $g_{ij}(z,t) \in k[z,t]$.

We now prove the result by contradiction. 
Suppose, if possible, that $A^{\phi} \nsubseteq k[x, z, t]$. 
Let $f \in A^{\phi} \setminus k[x, z, t]$. 
From the expression (\ref{eq2}) and the weights defined on the generators of $A$,
we have $\hat{f} = x^a y^b f_1(z,t)$  $(\in A^{\hat{\phi}})$ 
for some $0 \le a < m$, $b>0$ and $f_1(z,t) \in k[z,t]$. 
Since $A^{\hat{\phi}}$ is factorially closed in $A$ 
(cf. Lemma \ref{exp3} (i)), it follows that $y \in A^{\hat{\phi}}$.

Write the integer $s$ as $qp^r$, where $p \nmid q$. 
Since $p^e \nmid sp$, we have $e-r-1>0$
and since $sp \nmid p^e$, we have $q >1$.
We note that $k[x,x^{-1}, z, t]$ has the structure of a $\bZ$-graded algebra over 
$k$, say $k[x,x^{-1}, z, t] = \bigoplus_{i \in \bZ}C_i$, 
with the following weights on the generators
$$
\w (x) = 0, ~~ \w (z)=q, ~~ \w (t)=p^{e-r-1}.
$$
Consider the proper $\bZ$-filtration $\{A_n\}_{n \in \bZ}$ on $A$ defined by 
$A_n := A \cap \bigoplus_{i \le n}C_i$.  
Let $B$ denote the graded ring $\gr (A)(:= \bigoplus_{n \in \bZ}A_{n}/ A_{n-1})$
with respect to the above filtration. 
We now show that 
\begin{equation}\label{isom}
B \cong k[X,Y,Z,T]/(X^mY + Z^{p^e} + T^{sp}). 
\end{equation}
For $g \in A$, let $\bar{g}$ denote the image of $g$ in $B$.
%Each element $ f \in A$ can be uniquely written as 
%$$
%f = \sum_{n \ge 0} P_n( z, t ) x^n + \sum_{r >0} Q_{rj}(z, t) x^j y^r,
%$$
%where $P_n(z,t) , Q_{rj}(z,t) \in k[z,t] $, $ 0 \le j < m$. 
From the unique expression of any element $g \in A$ as given in (\ref{eq2}), 
it can be seen that the filtration defined on $A$ is admissible with the
generating set $\Gamma := \{x, y, z, t\}$. Hence
$B$ is generated by $\bar{x}$, $\bar{y}$, $\bar{z}$ and $\bar{t}$ 
(cf. Remark \ref{grmap} (2)). 

As $B$ can be identified with a subring of 
$\gr (k[x,x^{-1}, z,t]) \cong k[x,x^{-1}, z,t]$, we see that 
the elements $\bar{x}$, $\bar{z}$ and $\bar{t}$ of $B$
are algebraically independent over $k$.

Set $\ell_1 := qp^e$ and $\ell_2 := p^{e-r-1}$. Then $\ell_2 < \ell_1$; and
$x^my$, $z^{p^e}$ and $t^{sp}$ belong to $A_{\ell_1}\setminus A_{\ell_1-1}$ 
and $ t \in A_{\ell_2}\setminus A_{\ell_2-1}$.
Since $x^my + z^{p^e} + t^{sp} = -t \in A_{\ell_2}\subseteq A_{\ell_1-1}$, we have 
${\bar{x}^m\bar{y} + \bar{z}^{p^e} + \bar{t}^{sp}}= 0$ in $B$ (cf. Remark \ref{grmap} (1)). 

Now as $k[X,Y,Z,T]/(X^mY + Z^{p^e} + T^{sp})$ is an integral domain, 
the isomorphism in (\ref{isom}) holds. 
Therefore, by Theorem \ref{MDH}, $\hat{\phi}$ induces a non-trivial 
exponential map $\bar{\phi}$ on $B$ 
and $\bar{y}\in B^{\bar{\phi}}$, a contradiction by Lemma \ref{l1}.
Hence the result.
\end{proof}

\begin{rem}
{\em Lemma \ref{l1} and Proposition \ref{p1} do not hold for the case $m =1$, i.e., if 
$A= k[X,Y,Z,T]/(XY+ Z^{p^e}+T+ T^{sp})$. 
In fact, for any $f(y, z) \in k[y,z]$, setting
\[
\phi_1(y) =y, ~~ \phi_1(z) = z, ~~ \phi_1(t) =t+yf(y,z)U ~~ \text{and}~~
\phi_1(x) = x- f(y,z)U - ((t+yf(y,z)U)^{sp}- t^{sp})/y,
\]
we have a non-trivial exponential map $\phi_1$ on $A$ such that $y, z \in A^{\phi_1}$.
Similarly, for any $g(y, t) \in k[y,t]$, setting
\[
\phi_2(y) =y, ~~ \phi_2(t) = t, ~~ \phi_2(z) =z+yg(y,t)U ~~ \text{and}~~
\phi_2(x) = x- y^{p^e-1}g(y,t)^{p^{e}}U^{P^e},
\]
we get a non-trivial exponential map $\phi_2$ on $A$ such that $y, t \in A^{\phi_2}$.
Now interchanging $x$ and $y$ in $\phi_1$ and $\phi_2$, it is easy to see that 
there exist non-trivial exponential maps such that $x$ lies in their ring of invariants.
Hence ${\rm DK} (A) = A$.
}
\end{rem}

\begin{thm}\label{ce}
Let $k$ be any field of characteristic $p (> 0)$ and
$$
A= k[X,Y,Z,T]/(X^m Y + Z^{p^e} + T+ T^{sp}),
$$
where $m, e, s$ are positive integers such that $p^e \nmid sp$, 
$sp \nmid p^e$ and $m >1$.  Then $A \ncong_k k^{[3]}$.
\end{thm}

\begin{proof}
The result follows from Lemma \ref{r1} and Proposition \ref{p1}. 
\end{proof}

\begin{cor}\label{cor}
The Cancellation Conjecture does not hold for the polynomial ring 
$k[X,Y,Z]$, when $k$ is a field of positive characteristic. 
\end{cor}

\begin{proof}
 Follows from Theorem \ref{Ae} and Theorem \ref{ce}.
\end{proof}

\medskip

\noindent
{\bf Acknowledgement:} The author thanks Professors Shrikant M. Bhatwadekar, Amartya K. Dutta 
and Nobuharu Onoda for carefully going through the earlier draft and 
suggesting improvements.

\end{document}